\documentclass[12pt]{amsart}
\usepackage{amscd,amssymb}
\usepackage[arrow,matrix,graph]{xy}

\newtheorem{theorem}{Theorem}[section]

\newtheorem{lemma}[theorem]{Lemma}
\newtheorem{prop}[theorem]{Proposition}

\theoremstyle{definition}
\newtheorem{definition}[theorem]{Definition}

\newtheorem{remark}[theorem]{Remark}
\newtheorem*{ack}{Acknowledgment}

\def\id{{\rm id}}

\def\Prim{{\rm Prim}}
\def\gr{{\rm gr}}

\long\def\comment#1\endcomment{}

\def\bbZ{{\mathbb Z}}
\def\bbQ{{\mathbb Q}}
\def\bbK{{\mathbb K}}

\def\bbC{{\mathbb C}}

\def\bbL{{\mathbb L}}
\def\bbT{{\mathbb T}}
\def\bbF{{\mathbb F}}

\def\Aut{{\rm Aut}}
\def\AA{{\mathcal A}}
\def\BB{{\mathcal B}}

\def\LL{{\mathcal L}}

\def\PP{{\mathcal P}}
\def\PB{{\mathcal P\mathcal B}}
\def\BP{{\mathcal B\mathcal P}}
\def\BC{{\mathcal B\mathcal C}}
\def\OO{{\mathcal O}}

\def\UU{{\mathcal U}}

\def\Dd{Drinfel'd }

\def\ass{ associator }
\def\semass{ semi-associator }

\def\kab{ \bbK \langle \langle A,B \rangle \rangle}

\def\Alg{\widehat{\bbL}^{>1} \langle A,B \rangle}

\def\semk{\AA _n \rtimes \bbK [\Sigma _n]}

\def\semd{\AA _2 \rtimes \bbK [\Sigma _2]}
\def\sem3{\AA _3 \rtimes \bbK [\Sigma _3]}
\def\expsem{\exp\PP _n \rtimes \Sigma _n }

\def\a{\alpha }

\begin{document}

\title[Universal representations of braidlike groups]
{Universal representations of braid and braid-permutation groups}
      
\author[B. Berceanu]{Barbu Berceanu$^*$}
\address{Inst.~of Math.~Simion Stoilow, 
P.O. Box 1-764,
RO-014700 Bucharest, Romania}
\email{Barbu.Berceanu@imar.ro}

\author[\c S. Papadima]{\c Stefan Papadima$^*$}
\address{Inst.~of Math.~Simion Stoilow, 
P.O. Box 1-764,
RO-014700 Bucharest, Romania}
\email{Stefan.Papadima@imar.ro}

\thanks{$^*$Work partially supported by the CERES Programme of the
Romanian Ministry of Education and Research, contract 4-147/2004.}



\begin{abstract} 
Drinfel'd used associators to construct families of universal representations of
braid groups. We consider semi-associators (i.e., we drop the pentagonal axiom
and impose a normalization in degree one). We show that the process may be reversed, 
to obtain semi-associators from universal representations of $3$--braids.
We view braid groups as subgroups of braid-permutation groups. We construct a family 
of universal representations of braid-permutation groups, without using associators.
All representations in the family are faithful, defined over $\bbQ$ by simple explicit 
formulae. We show that they give universal Vassiliev--type invariants for
braid-permutation groups.
\end{abstract}

\maketitle

\section{Introduction}
\label{intro}

\subsection{}
\label{ss11}

In the foundational paper \cite{Drin}, \Dd  introduced and proved the existence of
associators, that is, formal series in two noncommutative variables with
coefficients in a characteristic zero field $\bbK$, satifying certain axioms.
These objects constitute the core structure leading to many important results.
They play a key role in the quantization of universal enveloping algebras.
There is a deep connection between associators and the absolute Galois group.
They appear in an essential way in the construction of universal finite type
invariants in low-dimensional topology. See \cite{Drin}, and also 
Birman's survey \cite{B} and the monograph \cite{Kas} by Kassel.

\subsection{}
\label{ss12}

One also finds in \cite{Drin} a bridge between associators and 
{\em universal representations} of Artin braid groups into braid algebras.
These algebras (defined over $\bbZ$) are semidirect products of symmetric group algebras 
$\bbK [\Sigma_n]$
and infinitesimal Artin algebras $\AA_n$ (alias 
complete algebras of horizontal chord diagrams on $n$ strings). Given an associator, \Dd
constructs a family of so-called universal representations of the braid groups into 
the corresponding braid algebras, satisfying certain natural properties. See sections
\ref{not} and \ref{sec:kaz} for details.

We consider in this note semi-associators, i.e., associators not required to verify the 
pentagonal axiom, which are still normalized in degree one; see Definition \ref{def:semiass}.

\begin{theorem}
\label{thm:main}
There is a natural bijection between semi-associators and universal representations of
$3$-braids.
\end{theorem}

In the above bijection, the universal $3$--representation coming from an associator
coincides with the one constructed in \cite{Drin}; see Theorem \ref{thm:dim3} and 
Remark \ref{rem:drincomp}. In other words, the \Dd approach may also be used to obtain
semi-associators from universal representations of $3$--braids.

\subsection{}
\label{ss13}

The \Dd representations from \cite{Drin} are faithful, as follows from work by Kohno 
\cite{K2}. There is however a practical inconvenient. Known explicit formulae for
$\bbC$--associators involve complicated multiple zeta values, see for instance
\cite[Ch. XIX]{Kas}. Over $\bbQ$, there is no example of explicitly described associator,
to our best knowledge. To remedy this, we view Artin braid groups $\BB_n$ inside the
{\em welded braid groups} $\BP_n$ introduced and studied by Fenn--Rimanyi--Rourke \cite{FRR}
(also known as {\em braid-permutation groups}).

We propose analogs of braid algebras (also defined over $\bbZ$), in this enlarged context,
namely {\em oriented braid algebras}; see Definitions \ref{def:braidoralg} and \ref{def:inforalg}.
These are semidirect product algebras, $\OO_n \rtimes \bbQ [\Sigma_n]$, where the
{\em oriented Artin algebra} $\OO_n$ is a cousin of $\AA_n$, obtained from {\em oriented}
horizontal chord diagrams. There is also a natural notion of universal family of
representations, of either welded or Artin braid groups into oriented braid algebras;
this notion is defined by conditions \eqref{n1} and \eqref{n3} from Theorem \ref{thm:newuniv}.
Our next goal is to point out similarities between universal representations of braid and
braid-permutation groups, into braid algebras (respectively oriented braid algebras), 
emphasizing the fact that the latter exhibit simpler qualitative properties.
We prove in Theorem \ref{thm:newuniv} the following.

\begin{theorem}
\label{thm:intro2}
There is a universal family of faithful representations, for both welded and Artin braid groups,
into oriented braid algebras. These representations are defined by explicit formulae,
over $\bbQ$.
\end{theorem}

The key point in the construction of associators and \Dd representations is the
analysis of the corresponding KZ--monodromy. The required flatness of KZ connection
forms is intimately related to the $1$--{\em formality} property of pure braid groups,
in the sense of D. Sullivan \cite{S}; see Section \ref{sec:mccool} for details.

The analogs of pure braid groups, in the context of welded braids, are McCool's groups from
\cite{MC}. Our key step in proving Theorem \ref{thm:intro2} is to show that the
McCool groups are $1$--formal, and to compute their rational Lie algebras associated to
the lower central series filtration. This is done in Theorem \ref{thm:mcformal}. 
See also Cohen--Pakianathan--Vershinin--Wu \cite{CPVW} for related results, in
particular for the determination of the integral associated graded Lie algebra of
upper-triangular McCool groups. New information on upper-triangular McCool groups
(which are proper subgroups of McCool groups)
may be found in Remark \ref{rem:upper}.

\subsection{}
\label{ss14}

It is well-known that \Dd representations have the following geometric interpretation,
see for instance \cite[Section 1]{Pa4}. Consider on rational group rings of Artin braid groups
the multiplicative {\em Vassiliev filtration}, obtained by resolving singularities 
of singular braids. On braid algebras, there is the natural multiplicative filtration
coming from the complete filtration of infinitesimal Artin algebras $\AA_n$. \Dd
representations are {\em universal finite type invariants} for the corresponding Artin braid groups.
This means that their canonical extensions to group rings respect the above filtrations,
and induce a multiplicative isomorphism at the associated graded level.

Our explicit representations of welded braid groups into oriented braid algebras from
Theorem \ref{thm:intro2} have the same geometric flavour. On rational group rings of
welded braids, we consider the multiplicative filtration obtained by resolving 
singularities (welds), as explained in \S \ref{ss62}. The complete filtration of
oriented Artin algebras $\OO_n$ naturally induces another multiplicative filtration,
on oriented braid algebras. The result below is proved in Section \ref{sec:new}.

\begin{theorem}
\label{thm:intro3}
The canonical extensions to group rings of the representations of welded braids into
oriented braid algebras, constructed in Theorem {\em \ref{thm:intro2}}, respect the
filtrations described above, and induce a multiplicative isomorphism 
at the associated graded level.
\end{theorem}

It is straightforward to check that the inclusion of group rings, 
$\bbQ [\BB_n] \hookrightarrow \bbQ [\BP_n]$, respects the abovementioned filtrations.
Consequently, the universal representations from Theorem \ref{thm:intro2},
$\bbQ [\BP_n] \rightarrow \OO_n \rtimes \bbQ [\Sigma_n]$, give finite type invariants for braids,
by restriction to $\bbQ [\BB_n]$. For $n\le 3$, it turns out that all finite type invariants
for braids arise in this way. See Section \ref{sec:new} for more details, including a
quantitative comparison between the associated graded algebras, $\AA_n^*$ and $\OO_n^*$.

\section{Dramatis personae}
\label{not}

Let us present the objects that inspired our study: braid groups and braid algebras.

\subsection{}
\label{ss21}

The {\em braid} and {\em pure braid} groups were 
defined and studied by E. Artin~\cite{Art}. The geometric braid group 
on $n$ strings is isomorphic to  
the group $ \BB _n $ generated by $\sigma _1, \dots,\sigma _{n-1}$, with 
defining relations
\begin{equation}
\label{eq11}
\left \{ 
\begin{array}{lcl}
\sigma _i\sigma _j=\sigma _j \sigma _i \, ,  &  {\rm for} & 2\leq |i-j|\, ;  \\
\sigma _{i+1}\sigma _i\sigma _{i+1}=\sigma _i\sigma _{i+1} \sigma _i \, ,  & 
{\rm for} & 1\leq i\leq n-2 
\end{array} 
\right .   
\end{equation}            
Using the natural morphism onto the symmetric group $\Sigma_n$,
$ \sigma _i \mapsto s_i :=(i,i+1) $, Artin identified
its kernel with the subgroup of {\em  pure braids}:
\begin{equation}
\label{eq:bsigma}
1\to \PB _n\to \BB _n\to \Sigma _n \to 1 \, . 
\end{equation}
As a set of generators for $ \PB _n $ we choose the elements
$$ \a_{ji}=\sigma _{i-1}\sigma _{i-2}\cdots \sigma _{j+1}\sigma _j^2
\sigma _{j+1}^{-1}
\cdots \sigma _{i-2}^{-1}\sigma _{i-1}^{-1}\, , \quad 1\leq j<i\leq n\, .$$
We also have a canonical embedding
\begin{equation}
\label{eq:bstab}
\BB _{n-1}\hookrightarrow \BB _n \, ,  \mbox{ given by } \sigma _i 
\mapsto \sigma _i \, , i=1,\ldots ,n-2  \, , 
\end{equation}
and an obvious stability property, expressed by the commuting diagram
$$ \begin{array}{ccccccccc}
1 & \to & \PB _{n-1} & \to & \BB _{n-1} & \to & \Sigma _{n-1} & \to & 1 \\
  &     & \downarrow      &     & \downarrow &     & \downarrow    &     &   \\
1 & \to & \PB _n     & \to & \BB _n     & \to & \Sigma _n     & \to & 1 
\end{array} $$
See also \cite{Mor} for complete proofs.

\subsection{}
\label{ss22}

To define the corresponding braid algebras, we  need 
{\em complete topological Hopf algebras} and 
{\em complete topological Lie algebras}; see \cite{Q} for details.  
Consider the tensor algebra, i.e., the free associative algebra 
with coefficients in a field $ \bbK $ of characteristic $ 0 $ ,
$ \bbT _{\bbK }\langle A_1,\ldots ,A_n\rangle $, where deg $A_i =1$ and the 
comultiplication is defined by 
$ \Delta A_i= A_i\otimes 1+1\otimes A_i$. 
We denote by $ \bbK \langle \langle A_1,\ldots ,A_n\rangle \rangle $
the completion of this Hopf algebra, with respect to the degree filtration. 
Its primitive part is the complete free Lie algebra,
$ \widehat{\bbL} \langle A_1,\ldots ,A_n\rangle =
\Prim \,\bbK \langle \langle A_1,\ldots ,A_n\rangle \rangle$.
We denote by $ \widehat{\bbT} ^{>k}\langle A_1,\ldots ,A_n\rangle $ and 
$ \widehat{\bbL} ^{>k}\langle A_1,\ldots ,A_n\rangle $ the corresponding complete 
filtrations. We denote the congruence modulo these ideals by $ \equiv _{k+1} $;
for instance, $ f\equiv _{2}g $ means that $ f $ and $ g $ have the same
linear part and the same constant term, if 
$f,g \in \bbK \langle \langle A_1,\ldots ,A_n\rangle \rangle $.
Similar considerations apply to arbitrary complete Hopf and Lie algebras.

Kohno~\cite{Kohn} and \Dd~\cite{Drin} introduced
infinitesimal versions of (pure) braid groups.
\begin{definition}[\cite{Kohn}]
\label{brinf}
The {\em infinitesimal Artin Hopf algebra} is the complete Hopf algebra 
given by the presentation:
$$ \AA _n \!=\! \bbK \langle \langle  t_{ij}=t_{ji} , 1\leq i\neq j\leq n
\mid [t_{ij},t_{ik}+t_{jk}]=0,[t_{ij},t_{kl}]=0 \mbox{ if } \{ i,j\}\cap
\{ k,l\} \! = \! \emptyset \rangle \rangle , $$
where $[u,v]:= uv-vu$ denotes the algebra commutator.
The {\em infinitesimal Artin Lie algebra} is the complete Lie algebra 
given by the presentation:
$$ \PP _n=\widehat {\bbL} \langle  t_{ij}=t_{ji} , 1\leq i\neq j\leq n
\mid [t_{ij},t_{ik}+t_{jk}]=0,[t_{ij},t_{kl}]=0 \mbox{ if } \{ i,j\}\cap
\{ k,l\}=\emptyset \rangle $$
\end{definition}

These two algebras determine each other: 
$  \AA _n=\UU \PP _n$  and $\PP _n=\Prim ( \AA _n) $.

There is a natural left action of the symmetric group $ \Sigma _n $ on 
the above algebras, defined by 
$ \pi (t_{ij})=t_{\pi (i)\pi (j)} $.
We shall use the exponential notation $ \pi (\Phi):= {}^{\pi }\Phi $.
For instance, if $ \pi =ijk:=
 \left ( \begin{array}{ccc} 
           1 & 2 & 3 \\
           i & j & k
          \end{array}
 \right ) \in  \Sigma _3 $, we denote $ \pi (\Phi)  $
by $ \mbox{}^{ijk}\Phi $.

\begin{definition}[\cite{Drin}]
\label{def:braidalg}
The {\em braid algebra} $ \AA _n \rtimes \bbK [\Sigma _n] $ is the 
semidirect algebra product, that is, $ \AA _n \otimes \bbK [\Sigma _n] $, 
with twisted  multiplication given by 
$ (a\otimes x)\cdot (b\otimes y)=a\cdot \mbox{}^xb\otimes xy $.
\end{definition}

The algebra $ \AA _n \rtimes \bbK [\Sigma _n] $ 
contains as  a multiplicative subgroup
the semidirect group product $ \expsem $. We thus
have split exact sequences of groups,
$$ 1\to \exp\PP _n\to \exp \PP _n\rtimes \Sigma _n \to  
\Sigma _n \to 1 \, , $$
together with compatible canonical  embeddings,
$$ \exp\PP _{n-1}\rtimes \Sigma _{n-1} \hookrightarrow 
\exp\PP _n\rtimes \Sigma _n \, .$$

\subsection{}
\label{ss23}
 
We will start by considering families of representations, 
$$ \rho _n:\BB _n\to  \AA _n \rtimes \bbK [\Sigma _n]\, ,$$
satisfying the following four natural properties, extracted from the work of \Dd \cite{Drin}.
\begin{itemize}
\item  {\bf Exponential type.} The representation $ \rho _n $ 
factorizes through the exponential subgroup:
  \begin{center} 
  \begin{picture}(300,80)
     \put(-10,38){$ {\bf (E)} $}
     \put(175,3){$ \AA _n \rtimes \bbK [\Sigma _n] $}
     \put(35,60){$ \BB _n $} \put(170,63){$ \exp \PP _n\rtimes \Sigma _n $}
     \put(63,65){\vector(1,0){100}} \put(100,70){$ \rho _n $} \put(100,27){$ \rho _n $}
     \put(195,55){\vector(0,-1){35}} \put(60,60){\vector(2,-1){100}}
  \end{picture}
  \end{center}
\item  {\bf Symmetry.} The diagram below commutes:
  \begin{center} 
  \begin{picture}(300,80)
     \put(35,3){$ \Sigma _n $} \put(60,5){\line(1,0){110}} \put(-10,38){$ {\bf (\Sigma )} $}
     \put(185,3){$ \Sigma _n $} \put(60,10){\line(1,0){110}}
     \put(35,63){$ \BB _n $} \put(170,63){$ \exp \PP _n\rtimes \Sigma _n $}
     \put(60,65){\vector(1,0){100}} \put(100,70){$ \rho _n $}
     \put(40,55){\vector(0,-1){35}} \put(190,55){\vector(0,-1){35}}
  \end{picture}
  \end{center}
\item  {\bf Stability.} One has commutative diagrams
  \begin{center} 
  \begin{picture}(300,80)
     \put(35,3){$  \BB _n $} \put(60,5){\vector(1,0){100}}
     \put(170,3){$  \exp \PP _n\rtimes \Sigma _n $}
     \put(100,10){$ \rho _n $}  \put(-10,38){$ {\bf (S)} $}
     \put(35,63){$ \BB _{n-1} $} 
     \put(170,63){$ \exp \PP _{n-1}\rtimes \Sigma _{n-1} $}
     \put(60,65){\vector(1,0){100}} \put(100,70){$ \rho _{n-1} $}
     \put(40,55){\vector(0,-1){35}} \put(190,55){\vector(0,-1){35}}
  \end{picture}
  \end{center}
\item {\bf Normalization.} The images of the generators $ (\sigma _i)_{i=1,n-1} $
  satisfy
$$ {\bf (N) } \quad \, \, \quad \rho _n \sigma _i=u_i\otimes s_i\, , \quad
\mbox{where} \quad u_i \in \exp \PP_n \quad \mbox{and} \quad 
u_i\equiv_2 1+\frac{t_{i,i+1}}{2} \, . $$ 
\end{itemize}

\begin{remark}
\label{orm}
From (E), $(\Sigma)$ and  (N), we obtain the 
following formula for the pure braid generators (abbreviating from now on
$u\otimes \id$ to $u$):
$$ \rho _n \a_{ji}\equiv_2 1+t_{ij}  \, .$$
Indeed, 
$ \rho_n \sigma _i^2=u_i\cdot ^{s_i}u_i\equiv_2 
(1+\frac{t_{i,i+1}}{2})^2\equiv_2 1+t_{i,i+1} $. Hence, 
$\rho_n \a_{ji} = us \cdot \rho_n \sigma_j^2 \cdot s^{-1}u^{-1}
\equiv_2 {}^s \rho_n \sigma_j^2$, where $u\in \exp \PP_n$, 
$s= s_{i-1}\cdots s_{j+1}$ and $\rho_n (\sigma_{i-1}\cdots \sigma_{j+1})= us$.
It is easy to check that ${}^s \rho_n \sigma_j^2 \equiv_2 1+t_{ji}$, 
which proves our claim.
\end{remark}

\section{\Dd  associators  and \Dd  representations}
\label{sec:kaz}

Following \Dd \cite{Drin}, we recall a method for constructing representations
having the four properties described in Section \ref{not}, based on the
notion of {\em associator}.

\subsection{}
\label{ss31}

The complete Hopf algebra $\bbK \langle \langle A, B \rangle \rangle$ has a
natural involution $s$: ${}^s A= B$, ${}^s B=A$. Given 
$\Phi \in \bbK \langle \langle A, B \rangle \rangle$, set
$\Phi_t := \Phi (t_{12}, t_{23})\in \AA_3$.

\begin{definition}
[~\cite{Drin}]
\label{assoc}
An element $ \Phi \in \kab $ 
is called an {\em associator} if it 
satisfies the following conditions :
$$ \begin{array}{lc}
{\bf (AE) } &     \Phi = \exp(\varphi )\, , \mbox{ with } \varphi \in \Alg \\
{\bf (AS) } &     \mbox{}^s\Phi = \Phi ^{-1} \\
{\bf (H1) } &   \exp (\frac {t_{12}+t_{13}}{2})=\mbox{}^{231}\Phi _t^{-1}
\cdot \exp (\frac {t_{13}}{2})\cdot \mbox{}^{213}\Phi _t \cdot
 \exp (\frac {t_{12}}{2})\cdot \Phi _t^{-1} \\
{\bf (H3) } &  \! \! \! \exp (\frac {t_{13}+t_{23}}{2})=\mbox{}^{312}\Phi _t \cdot
 \exp (\frac {t_{13}}{2})\cdot \mbox{}^{132}\Phi _t^{-1} \cdot 
\exp (\frac {t_{23}}{2})\cdot \Phi _t \\
{\bf (P) }  & \!  \Phi (t_{12},t_{23}+t_{24})\Phi (t_{13}+t_{23},t_{34})=
   \Phi  (t_{23},t_{34})\Phi (t_{12}+t_{13},t_{24}+t_{34})
   \Phi (t_{12},t_{23})
\end{array}  $$
\end{definition}

The first two equations must hold in $ \kab $, the next two in $\AA_3$,
and the last in $ \AA _4$.
\Dd  proved that associators exist; see \cite[Proposition 5.4]{Drin}.

\begin{definition}
\label{def:semiass}
We will say that $ \Phi \in \kab $ 
is a {\em semi-associator} if it 
satisfies the properties from Definition \ref{assoc}, except (P).
\end{definition}

Let us remark that, in the definition of a \semass ,
(H1) and (H3) are equivalent; see \cite[p. 848]{Drin}.

\subsection{}
\label{ss32}

\Dd ~\cite{Drin} used associators to construct families of representations, 
$ \{ \rho _n:\BB _n\to \semk \}$, defined by the formulae:
\begin{equation}
\label{eq:kzrep}
\left \{ 
\begin{array}{lll}
        \sigma _1 & \mapsto &  \exp (\frac{t_{12}}{2})\otimes s_1 \, ; \\
        \sigma _{i>1} & \mapsto &  \Phi (\sum_{j<i}t_{ji}, t_{i,i+1})^{-1}\cdot
         (\exp (\frac{t_{i,i+1}}{2})\otimes s_i)\cdot 
             \Phi (\sum_{j<i}t_{ji}, t_{i,i+1})
\end{array} 
\right. 
\end{equation}

\begin{theorem}[~\cite{Drin},~\cite{Piun}]
\label{thm:repd}
If $ \Phi \in \kab$ is an \ass , the above formulae define representations
$$ \rho _n :\BB _n\to \semk \, , $$
satisfying the properties {\em (E), $(\Sigma )$, (S), (N)} from Section {\em \ref{not}}.
\end{theorem}

\begin{proof}
The only new claim concerns the normalization property. This may be checked as follows.
First, it is straightforward to deduce from \eqref{eq:kzrep} that 
$\rho_n \sigma_i = u_i \otimes s_i$, where 
\begin{equation}
\label{eq:udrin}
u_i= \Phi (\sum_{j<i}t_{ji}, t_{i,i+1})^{-1}\cdot
         \exp (\frac{t_{i,i+1}}{2})\cdot 
         {}^{s_i} \Phi (\sum_{j<i}t_{ji}, t_{i,i+1})\, .
\end{equation}
Axiom (AE)  implies that $\Phi\equiv_2 1$, which yields property (N).
\end{proof}

Another important feature is that all \Dd representations \eqref{eq:kzrep} 
are faithful. This follows from ($\Sigma$), \cite[p. 848]{Drin}, and 
\cite[Proposition 1.3.3]{K2}.

\section{Universal representations of $ 3 $-braids}
\label{sec:rep3}

We may now give the proof of Theorem \ref{thm:main}.

\subsection{}
\label{ss41}

Given a  representation $ \rho :\BB _3\to \sem3 $, denote by $\rho'$ 
its restriction to $\BB_2$ (embedded in $\BB_3$ as explained in \S\ref{ss21}).

\begin{definition}
\label{def:urepb}
A representation $\rho$ as above is called {\em universal} if the family
$\{ \rho, \rho' \}$ satisfies properties (E), $(\Sigma)$, (S) and (N) from
\S\ref{ss23}. 
\end{definition}

The set of universal representations of $\BB_3$ will be denoted by
$\UU rep\, (\BB_3)$.

\begin{lemma}
\label{thm:dim2}
There is a unique representation $ \rho' :\BB _2\to \semd $ that satisfies 
conditions {\em (E), $(\Sigma)$} and {\em (N)}, given by  
$\rho' \sigma _1= \exp (\frac {t_{12}}{2})\otimes s_1$.
\end{lemma}

\begin{proof}
Properties (E) and $(\Sigma)$ together are saying that 
$\rho' \sigma _1= \exp (\lambda t_{12})\otimes s_1 $, 
with $ \lambda \in \bbK$. By (N), $\lambda=\frac {1}{2} $.
\end{proof}

We parametrize representations $ \rho :\BB _3\to \sem3 $, using another 
presentation of $ \BB _3 $, derived from \eqref{eq11}:
\begin{equation}
\label{eq:gpres}
\BB _3=\langle \, \sigma _1, \Delta \mid \sigma _1\Delta \sigma _1=
\Delta \sigma _1^{-1}\Delta \, \rangle\, , 
\end{equation}
where $ \Delta =\sigma _1\sigma _2\sigma _1 $.
The fundamental element $ \Delta $ has the property that its 
square $ \Delta ^2 $ generates the center of the braid group~\cite{Chow}.
We also consider the element of $  \PP _3 $, $ T=\frac{1}{2}(t_{12}+t_{13}+t_{23})$ 
and observe that
\begin{equation}
\label{eq:ppres}
\PP _3=\bbK \cdot T\times \widehat {\bbL} \, \langle \,A:=t_{12},B:=t_{23} \, \rangle  \, ,
\end{equation}
as Lie algebras.
In particular, the center of the Lie algebra $ \PP _3 $ is $ \bbK \cdot T $.

\subsection{}
\label{ss42}

Let $\Psi =\exp(\psi)$ be a group-like element of $\kab$. Set
\begin{equation}
\label{eq42}
\left \{ 
\begin{array}{l} 
\rho \sigma _1=\exp (\frac {t_{12}}{2})\otimes s_1 
\in \exp (\PP_2) \rtimes \Sigma_2 \\
\rho \Delta =\exp (T)\cdot \Psi _t^{-1}\otimes 321
\in \exp (\PP_3) \rtimes \Sigma_3
\end{array}
\right .
\end{equation}

Theorem \ref{thm:main} is a consequence of the following result.

\begin{theorem}
\label{thm:dim3}
If $\Psi$ is a \semass (in the sense of {\em Definition \ref{def:semiass}}), then
\eqref{eq42} defines a universal representation $\rho \in \UU rep\, (\BB_3)$
(in the sense of {\em Definition \ref{def:urepb}}). Conversely, every universal
representation of $\BB_3$  has a unique parametrization of the form \eqref{eq42}, 
where $\Psi$ is a \semass.
\end{theorem}
 
We start by noting that Definition \ref{def:urepb}, Lemma \ref{thm:dim2} and presentation
\eqref{eq:gpres} readily imply that $\rho \in \UU rep\, (\BB_3)$ if
and only if $\rho$ is of the form \eqref{eq42}, with $\Psi_t$ replaced by 
$\Phi\in \exp(\PP_3)$, and the following two properties hold:
$$ {\bf (YB)} \quad \quad \rho \Delta = \rho \sigma_2 \cdot \rho \sigma_1 \cdot \rho \sigma_2$$
(expressing the fact that $\rho$ is a representation) and
$$ {\bf (N2)} \quad \quad \rho \sigma_2 =u_2 \otimes s_2\, ,\quad \mbox{with} \quad
u_2\equiv_2 1+ \frac{t_{23}}{2}\, .$$

\begin{lemma}
\label{lem:norm}
In the above setting, {\em (N2)} is equivalent to $\log \Phi \equiv_2 0$ in $\PP_3$.
\end{lemma}

\begin{proof}
Since $ \Delta =\sigma _1\sigma _2\sigma _1 $, 
$$ \rho \sigma _2=(\exp (-\frac {t_{12}}{2})\otimes s_1)\cdot (\exp T\cdot \Phi ^{-1}\otimes 321)
\cdot (\exp (-\frac {t_{12}}{2})\otimes s_1 )\, .$$
It follows that 
\begin{equation}
\label{eq:rho2}
\rho \sigma_2 = \exp (-\frac {t_{12}}{2})\cdot \exp T\cdot ^{213}\Phi ^{-1}\cdot
\exp (-\frac {t_{13}}{2})\otimes s_2 \, .
\end{equation}
Condition (N2) becomes $(1+ \frac{t_{23}}{2})\cdot ^{213}\Phi^{-1} 
\equiv_2 1+ \frac{t_{23}}{2}$, whence the result.
\end{proof}

By resorting to \eqref{eq:ppres}, we infer that $\Phi= \Psi_t$, with
$\Psi =\exp (\psi)$ and $\psi\in \Alg$. The proof of Theorem \ref{thm:dim3}
is thus reduced to showing that the Yang--Baxter equation (YB) for $\Phi$
is equivalent to properties (AS) and (H) from Definition \ref{assoc} for $\Psi$.
To prove this equivalence, we need a preliminary result.

\begin{lemma}
\label{l12}
Assume $\{ \mu_{ij}\in \AA_n \}_{1\le i<j \le n}$ are such that
$ \mu_{ij}\equiv_2 t_{ij} $. Denote by $V$ the $\bbK$--span of $\{ \mu_{ij}\}$.
Then, for an arbitrary $ t\in \AA _n $, we have an expansion
$ t=\sum _{k\geq 0}v_k$, with the property that $v_k\in V^k :=
\overbrace{V\cdots V}^{k}$, for $k>0$, and $v_0 \in V^0 :=\bbK \cdot 1$.
\end{lemma}

\begin{proof}
By induction on $ k $,
$$ 0\equiv_k t-\sum_{i<k}v_i\equiv_{k+1}\sum k-\mbox{monomials in }(t_{ij})\equiv_{k+1}
\sum k-\mbox{monomials in }(\mu_{ij}) \, , $$
which completes the induction step.
\end{proof}

Next, we use the above lemma to deduce the following.

\begin{lemma}
\label{lem:ybas}
Equation {\em (YB)} for $\Phi$ implies condition {\em (AS)} for $\Psi$.
\end{lemma}

\begin{proof}
Knowing that $\rho$ is a representation and $\Delta^2\in \PB_3$ is central in $\BB_3$,
we infer that $ [\rho \Delta ^2, \rho\a_{ij}]=0 $ in $\AA_3$, for $1\le i<j \le 3$.
We also know from Remark \ref{orm} that $\rho \a_{ij}= 1+ \mu_{ij}$, with
$\mu_{ij}\equiv_2 t_{ij}$. Due to Lemma \ref{l12}, we obtain that $\rho \Delta^2$
is central in $\AA_3$. Hence, $\rho \Delta^2= \exp(hT)$, for some $h\in \bbK$;
see \eqref{eq:ppres}. On the other hand, 
$$\rho \Delta ^2= \rho (\a _{13}\a _{23}\a _{12})\equiv_2 1+2T \, ,$$
again by Remark \ref{orm}. This forces $h=2$. Therefore,
$\exp(2T)= (\rho \Delta)^2= \exp(2T)\cdot \Phi^{-1}\cdot ^{321}\Phi^{-1}$,
as follows from \eqref{eq42}. Finally, since $\Phi =\Psi_t$, 
$^{321}\Phi= \Phi^{-1}$ translates to $^s \Psi= \Psi^{-1}$, as asserted.
\end{proof}

\subsection{}
\label{ss43}

Recall that axioms (H1) and (H3) from Definition \ref{def:semiass} are equivalent.
With this remark, the Lemma below will finish the proof of Theorem \ref{thm:dim3}.

\begin{lemma}
\label{lem:final}
The Yang--Baxter equation {\em (YB)} for $\Phi$ is equivalent to 
the hexagonal axiom {\em (H3)} for $\Psi$.
\end{lemma}

\begin{proof}
First, we may rewrite \eqref{eq:rho2} in the form
$$ \rho \sigma_2= \exp (\frac{t_{13}+t_{23}}{2})\cdot ^{213}\Phi^{-1}\cdot 
\exp (\frac{-t_{13}}{2})\otimes s_2\, ,$$
since $[t_{12}, t_{13}+t_{23}]=0$. Together with \eqref{eq42}, this leads to
$$\rho \sigma _2\cdot \rho \sigma _1\cdot \rho \sigma _2=
\exp (\frac{t_{13}+t_{23}}{2})\cdot 
^{213}\Phi ^{-1}\cdot \exp (\frac{t_{12}+t_{23}}{2})\cdot ^{132}\Phi ^{-1}\cdot 
\exp (-\frac{t_{23}}{2})\otimes 321 \, .$$
Comparing this with \eqref{eq42}, we find that (YB) is equivalent to
\begin{equation}
\label{eq:f1}
\exp (\frac{t_{12}}{2})\cdot \Phi^{-1}= 
^{213}\Phi ^{-1}\cdot \exp (\frac{t_{12}+t_{23}}{2})\cdot ^{132}\Phi ^{-1}\cdot 
\exp (-\frac{t_{23}}{2})\, ,
\end{equation}
or
\begin{equation}
\label{eq:f2}
\exp (\frac{t_{12}+t_{23}}{2})= ^{213}\Phi \cdot \exp (\frac{t_{12}}{2})\cdot \Phi ^{-1}
\cdot \exp (\frac{t_{23}}{2})\cdot ^{132}\Phi \, .
\end{equation}
Applying $s_2$ to \eqref{eq:f2}, we find the equivalent form
\begin{equation}
\label{eq:f3}
\exp (\frac{t_{13}+t_{23}}{2})= ^{312}\Phi \cdot \exp (\frac{t_{13}}{2})\cdot 
^{132}\Phi ^{-1} \cdot \exp (\frac{t_{23}}{2})\cdot \Phi \, ,
\end{equation}
which is precisely condition (H3) for $\Psi$.
\end{proof}

\begin{remark}
\label{rem:drincomp}
It is worth pointing out that the representation $\rho$ corresponding to
a semi-associator $\Psi$, in our Theorem \ref{thm:dim3}, is given by
\eqref{eq:kzrep}, as in the \Dd construction of representations coming from
an associator.

To check this, start with a group-like element $\Psi\in \kab$ and set 
$\Phi= \Psi_t$. For $i=2$, \eqref{eq:kzrep} gives
\begin{equation}
\label{eq:c1}
\rho \sigma_2= \Phi^{-1}\cdot \exp(\frac{t_{23}}{2})\cdot
^{132}\Phi \otimes s_2 \, .
\end{equation}
Comparing this with \eqref{eq:rho2}, we find that we must verify the equality
\begin{equation}
\label{eq:c2} 
\Phi^{-1}\cdot \exp(\frac{t_{23}}{2})\cdot ^{132}\Phi =
\exp(\frac{t_{13}+t_{23}}{2})\cdot ^{213}\Phi^{-1} \cdot \exp(-\frac{t_{13}}{2})\, ,
\end{equation}
that is,
\begin{equation}
\label{eq:c3}
\exp(\frac{t_{13}+t_{23}}{2})= \Phi^{-1}\cdot \exp(\frac{t_{23}}{2})\cdot
^{132}\Phi \cdot \exp(\frac{t_{13}}{2}) \cdot ^{213}\Phi \, .
\end{equation}
By applying $s_1$, we put \eqref{eq:c3} in the equivalent form
\begin{equation}
\label{eq:c4}
\exp(\frac{t_{13}+t_{23}}{2})= ^{213}\Phi^{-1}\cdot \exp(\frac{t_{13}}{2})\cdot
^{231}\Phi \cdot \exp(\frac{t_{23}}{2}) \cdot \Phi \, .
\end{equation}

Finally, \eqref{eq:c4} above is seen to coincide with axiom (H3) of
the semi-associator $\Psi$, by using the equality $^{321}\Phi= \Phi^{-1}$,
which corresponds to axiom (AS).
\end{remark}

\section{Braid-permutation and basis-conjugating groups}
\label{sec:mccool}

We examine analogs of the braid groups and braid algebras from Section \ref{not}.

\subsection{}
\label{ss51}

Denote by $\bbF_n$ the free group generated by $x_1, \dots, x_n$. The 
{\em braid-permutation group}, $\BP_n \subset \Aut (\bbF_n)$, was investigated 
in detail in \cite{FRR}. Its elements are the automorphisms $a\in \Aut (\bbF_n)$
acting by 
\begin{equation}
\label{eq:defbp}
a (x_i)= y_i^{-1}x_{s(i)}y_i\, , \quad \mbox{for} \quad 1\le i\le n\, ,
\end{equation}
where $y_i\in \bbF_n$ and $s\in \Sigma_n$, and the group product is composition of
automorphisms. The abelianization homomorphism, $\Aut (\bbF_n)\to \Aut (\bbZ^n)$,
induces a short exact sequence
\begin{equation}
\label{eq:bpsigma}
1\rightarrow \BC_n \rightarrow \BP_n \rightarrow \Sigma_n \rightarrow 1
\end{equation}
Unlike their braid analogs \eqref{eq:bsigma}, the above sequences are
naturally split by the permutation action of $\Sigma_n$ on $\{ x_1, \dots, x_n \}$.

The following presentation of $\BC_n$, also known as the {\em McCool group} of 
size $n$, was found in \cite{MC}. For $1\le i\ne j\le n$, denote by $a_{ij}$
the automorphism of $\bbF_n$ which sends $x_i$ to $x_j^{-1} x_i x_j$ and fixes $x_k$
for $k\ne i$. Then $\BC_n$ is generated by $\{ a_{ij} \}_{1\le i\ne j\le n}$, with
defining relations
\begin{equation}
\label{eq:relmc}
\left \{
\begin{array}{cll}
(I) & (a_{ik}, a_{jk})= 1\, , & \forall ~1\le i\ne j \ne k\le n\, ; \\
(II) & (a_{ij}, a_{ik} a_{jk})=1\, , & \forall ~1\le i\ne j \ne k\le n\, ; \\
(III) & (a_{ij}, a_{kl})= 1\, , & \forall ~1\le i\ne j \ne k \ne l\le n\, ,
\end{array}
\right .
\end{equation}
where $(x,y):= xyx^{-1}y^{-1}$ stands for the group commutator.
It is easy to check that 
\begin{equation}
\label{eq:equiv}
s a_{ij} s^{-1}= a_{s(i), s(j)}\, ,
\end{equation}
for all $1\le i\ne j\le n$ and $s\in \Sigma_n$. In condensed form,
$\BP_n =\BC_n \rtimes \Sigma_n$; see also \cite{CPVW}.

By Artin's theorem (see \cite[Theorem 1.9 on p.30]{Bi}), $\BB_n$ embeds 
onto the subgroup of elements in $\BP_n$ fixing $x_1\cdots x_n$.
As noted in \cite{FRR}, the embedding $\BB_n \hookrightarrow \BP_n$ is given by
\begin{equation}
\label{eq:sigmaas}
\sigma_i = a_{i, i+1} s_i\, , \quad \mbox{for} \quad 1\le i<n\, .
\end{equation}
Clearly, the sequences \eqref{eq:bsigma} and \eqref{eq:bpsigma} are compatible
with this embedding. In particular, $\PB_n \hookrightarrow \BC_n$.

We also have a canonical embedding, 
\begin{equation}
\label{eq:bpstab}
\BP_{n-1} \hookrightarrow \BP_n \, ,
\end{equation}
defined by $a (x_n)= x_n$, for $a\in \BP_{n-1}$, and commutative diagrams
\begin{equation}
\label{eq:stabcomp}
\begin{array}{ccc}
\BB_{n-1} & \longrightarrow & \BB_n \\
\downarrow & & \downarrow \\
\BP_{n-1} & \longrightarrow & \BP_n
\end{array}
\end{equation}
Finally, the projections of braid-permutation groups  onto symmetric groups from
\eqref{eq:bpsigma} are compatible with stabilization \eqref{eq:bpstab},
like in the case of braid groups; see the end of \S \ref{ss21}.

\subsection{}
\label{ss52}

We describe now the analogs of Artin and braid algebras from \S \ref{ss22}.

\begin{definition}
\label{def:inforalg}
The {\em oriented Artin Hopf algebra} is the complete Hopf algebra $\OO_n$ 
obtained from $\bbQ \langle \langle v_{ij}\; \mid \; 1\le i \ne j \le n \rangle \rangle$
by imposing the relations
\[
\left \{
\begin{array}{cll}
(I) & [v_{ik}, v_{jk}]= 0\, , & \forall ~1\le i\ne j \ne k\le n\, ; \\
(II) & [v_{ij}, v_{ik}+ v_{jk}]=0\, , & \forall ~1\le i\ne j \ne k\le n\, ; \\
(III) & [v_{ij}, v_{kl}]= 0\, , & \forall ~1\le i\ne j \ne k \ne l\le n\, .
\end{array}
\right .
\]
The {\em associated graded} oriented Artin Hopf algebra (with respect to the 
canonical complete filtration of $\OO_n$) is denoted by $\OO_n^* = \oplus_{k\ge 0} \OO_n^k$.
It is a Hopf algebra with grading, obtained as the quotient of $\bbT_{\bbQ} \langle v_{ij}\rangle$,
graded by tensor length, by the above relations (I)--(III).

The {\em oriented Artin Lie algebra} is the complete Lie algebra $\LL_n:= \Prim (\OO_n)$,
the quotient of $\widehat{\bbL} \langle v_{ij}\rangle$ by the relations (I)--(III).
The {\em associated graded} oriented Artin Lie algebra 
(with respect to the 
canonical complete filtration of $\LL_n$) is denoted by $\LL_n^* = \oplus_{k\ge 1} \LL_n^k$.
It is a Lie algebra with grading, obtained as the quotient of the free $\bbQ$--Lie algebra
$\bbL \langle v_{ij} \rangle$, graded by bracket length, by the relations (I)--(III).
\end{definition}

There is a natural left action of $\Sigma_n$ on the algebras $\OO_n= \UU \LL_n$ and $\LL_n$,
defined on generators by ${}^{\pi} v_{ij}= v_{\pi(i), \pi(j)}$, in exponential notation.

\begin{definition}
\label{def:braidoralg}
The {\em oriented braid algebra} $\OO_n \rtimes \bbQ [\Sigma_n]$ is the semidirect
algebra product, $\OO_n \otimes \bbQ [\Sigma_n]$, with respect to the above 
$\Sigma_n$--action on $\OO_n$ (where the twisted multiplication is as in Definition
\ref{def:braidalg}).
\end{definition}
 
The algebra $\OO_n \rtimes \bbQ [\Sigma_n]$ contains as a multiplicative subgroup 
the semidirect group product $\exp \LL_n \rtimes \Sigma_n$. We thus have 
split exact sequences of groups,
$$ 1\to \exp \LL_n \to \exp \LL_n \rtimes \Sigma_n \to \Sigma_n \to 1\, ,$$
together with compatible canonical embeddings,
$$\exp \LL_{n-1} \rtimes \Sigma_{n-1} \hookrightarrow  \exp \LL_n \rtimes \Sigma_n\, .$$

\subsection{}
\label{ss53}

We may now define our analogs of \Dd representations. Send 
$a_{ij}$ to $\exp (v_{ij})\in \exp \LL_n$,
for $1\le i\ne j\le n$. The defining Lie relations of $\LL_n$ from Definition \ref{def:inforalg}
readily imply that the defining group relations of $\BC_n$ from \eqref{eq:relmc} 
are respected. Consequently, we obtain a representation,
\begin{equation}
\label{eq:defr}
R_n \colon \BC_n \longrightarrow \exp \LL_n \subset \OO_n \, .
\end{equation}
It follows from \eqref{eq:bpsigma} and \eqref{eq:equiv} that
\begin{equation}
\label{eq:defrbp}
R_n \otimes \id \colon \BP_n \longrightarrow \exp \LL_n \rtimes \Sigma_n 
\subset \OO_n \rtimes \bbQ [\Sigma_n]
\end{equation}
is a representation of the braid-permutation group $\BP_n$ into the oriented braid algebra
$\OO_n \rtimes \bbQ [\Sigma_n]$, having the exponential property (E) from \S \ref{ss23}.

\subsection{}
\label{ss54}

As is well-known, the existence of representations, 
$\rho_n \colon \BB_n \to \AA_n \rtimes \bbQ [\Sigma_n]$, satisfying properties
(E), $(\Sigma)$ and (N) from \S \ref{ss23}, is intimately related to formality
properties of ordered configuration spaces of $\bbC$ and of their fundamental groups, $\PB_n$.
To obtain the same formality property for the McCool groups $\BC_n$, we turn to
a review of Malcev completion, following \cite[Appendix A]{Q}.

A {\em Malcev Lie algebra} is a rational Lie algebra $E$,
together with a complete, descending $\bbQ$-vector space filtration,
$\{F_rE\}_{r \ge 1}$, such that:
\begin{enumerate}
\item \label{eq:ml1}
$F_1E=E$;
\item \label{eq:ml2}
$[F_rE,F_sE]\subset F_{r+s}E$, for all $r$ and $s$;
\item \label{eq:ml3}
the associated graded Lie algebra, $\gr_F^*(E)=\bigoplus _{r\ge 1} 
F_rE/F_{r+1}E $,
is generated in degree~$*=1$.
\end{enumerate}
For example, the canonical complete filtration of $\LL_n$ makes it a Malcev Lie algebra,
with $\gr_F^* (\LL_n)=\LL_n^*$, as Lie algebras; see Definition \ref{def:inforalg}.

Let $G$ be a group.  The lower central series of $G$ is the sequence 
of normal subgroups $\{\Gamma_k G\}_{k\ge 1}$, defined inductively 
by $\Gamma_1 G=G$ and 
$\Gamma_{k+1}G =(\Gamma_k G,G)$. Observe that  
the successive quotients $\Gamma_k G/\Gamma_{k+1} G$ 
are abelian groups. The direct sum of these quotients, 
$\gr_{\Gamma}^*(G):=\oplus_{k\ge 1} 
\Gamma_k G/ \Gamma_{k+1} G$
is the {\em associated graded Lie algebra} of $G$. 
The Lie bracket is induced from the group commutator. Consequently,
$\gr_{\Gamma}^*(G)$ is generated as a Lie algebra by $\gr_{\Gamma}^1(G)$.

A group homomorphism, $\kappa\colon G\to \exp E$, where $E$ is a Malcev Lie algebra,
induces a degree zero morphism of Lie algebras, 
$\gr^*(\kappa)\colon \gr_{\Gamma}^*(G)\otimes \bbQ \rightarrow \gr_F^*(E)$. 
If $\gr^*(\kappa)$ is
an isomorphism, $\kappa$ is called a {\em Malcev completion} of $G$. There is a functorial
Malcev completion, $\kappa_G \colon G\to \exp E_G$, where $E_G$ is called the Malcev Lie
algebra of $G$. If $\kappa$ is another Malcev completion, there is an isomorphism of
complete Lie algebras, $f: E_G \stackrel{\sim}{\to} E$, such that 
$\kappa= \exp (f)\circ \kappa_G$. For example, the Malcev completion of $\bbF_n$, the
free group on $x_1, \dots, x_n$, is given by the tautological representation,
$\kappa_n\colon \bbF_n \to \exp \widehat{\bbL}\langle x_1, \dots, x_n \rangle$,
sending each $x_i$ to $\exp(x_i)$. 

Let $G= \langle x_1, \dots, x_n \mid w_1, \dots, w_m \rangle$ be a finitely presented group.
Define the Malcev Lie algebra $E$ to be the quotient of 
$\widehat{\bbL}\langle x_1, \dots, x_n \rangle$ obtained by imposing the relations
$\{ \log (\kappa_n(w_j))=0 \}$. It follows from \cite[Theorem 2.2]{Pa1} that
the tautological homomorphism,
\begin{equation}
\label{eq:fpmal}
\kappa\colon G\rightarrow \exp E\, ,
\end{equation}
is a Malcev completion.

Following D. Sullivan \cite{S}, we will say that a finitely presented group $G$
is {\em $1$--formal} if the Malcev Lie algebra $E_G$ is quadratic, that is,
obtainable from a free complete Lie algebra $\widehat{\bbL} \langle x_1, \dots, x_n \rangle$
by imposing homogeneous relations of bracket length two. 

The usual approach to the $1$-formality property and the construction of a Malcev
completion for pure braid groups uses the following ingredients. The starting remark is that
$\PB_n =\pi_1 F_n (\bbC)$, where $F_n (\bbC)$ denotes the ordered configuration space of
$n$ distinct points in $\bbC$. Next, note that $F_n (\bbC)$ is the complement to the
complex hyperplane arrangement of all diagonals $\{ z_i =z_j \}$ in $\bbC^n$. As such,
$F_n (\bbC)$ is a formal space in the sense of D. Sullivan \cite{S}, as follows from
basic results in arrangement theory, due to Orlik and Solomon \cite{OS}. Since the
fundamental group of a formal space is $1$-formal \cite{S}, $\PB_n$ is $1$-formal. 
Moreover, Chen's theory of iterated integrals \cite{Che} implies that the monodromy
representation of the canonical flat connection on the formal space $F_n (\bbC)$
is a Malcev completion homomorphism for $\PB_n$. Finally, it turns out that the KZ
connection coincides with the canonical flat connection. See also Kohno \cite{Kohn}.

We will take a completely different, much simpler, approach for the McCool groups
$\BC_n$. More precisely, we will exploit the particularly simple form of their
presentations \eqref{eq:relmc}. This will enable us to deduce the $1$-formality 
property, and to show that the explicit representations \eqref{eq:defr} are Malcev
completions, using only the definitions.

\subsection{}
\label{ss55}

It follows from the construction of the functorial Malcev completion $\kappa_G$
\cite[Appendix A]{Q} and a result on $I$--adic filtrations of group rings 
\cite[Proposition 2.2.1]{Che} that any Malcev completion $\kappa$ of a
{\em residually torsion-free nilpotent} group $G$ is faithful, if $G$ is
finitely generated. By definition, $G$ has the above residual property if all
non-trivial elements of $G$ are detected by homomorphisms $G\to N$, where $N$
is a torsion-free nilpotent group. Consequently, this property is inherited by
subgroups.

We are going to derive the faithfulness of our representations \eqref{eq:defr}
from a general result about residual torsion-free nilpotence of Torelli groups.
The {\em Torelli group} $T_G$ of a group $G$ is
\[
T_G := \{ a\in \Aut (G) \mid a\equiv \id \quad \mbox{mod} \quad \Gamma_2 G \}\, .
\]
It is endowed with the decreasing filtration 
\[
F_sT_G := \{ a\in \Aut (G) \mid a\equiv \id \quad \mbox{mod} \quad \Gamma_{s+1} G \}
\]
($s\ge 1$), which has the property that $\Gamma_s T_G \subset F_s T_G$, for all $s$.
 
\begin{prop}
[\cite{H}]
\label{prop:tore}
Assume $\cap_k \Gamma_k G= \{ 1\}$ and $\gr_{\Gamma}^*(G)$ is torsion-free.
Then the Torelli group $T_G$ is residually torsion-free nilpotent.
\end{prop}

\begin{proof}
The result is stated by Hain without proof in \cite[Section 14]{H}. For the
benefit of the reader, we are going to give a proof. By the first assumption on $G$,
any non-trivial element, $\id \ne a\in T_G$, is detected by the natural homomorphism,
$T_G \to T_{G_k}$, for some $k$, where $G_k:= G/\Gamma_k G$ is nilpotent and 
inherits the second hypothesis from $G$. It will be thus enough to assume that moreover
$G$ is nilpotent and to prove that in this case $T_G$ must be torsion-free nilpotent.
Nilpotence follows from the obvious fact that $F_{k-1}T_G =\{ \id \}$, if 
$\Gamma_k G= \{ 1\}$. Torsion-freeness may be verified by induction, as soon as 
we know that all quotients, $F_sT_G/F_{s+1}T_G$, are torsion-free. 
By \cite[Proposition 2.1]{Pa3}, the above Torelli filtration quotient embeds into
the degree $s$ derivations of the associated graded Lie algebra $\gr_{\Gamma}^*(G)$,
which has no torsion.
\end{proof}

Note that the groups $G$ from Proposition \ref{prop:tore} are themselves residually
torsion-free nilpotent. The hypotheses of the proposition are satisfied by free groups,
see \cite{MKS}. This implies the residual torsion-free nilpotence of $\BC_n \subset T_{\bbF_n}$;
see \eqref{eq:defbp} and \eqref{eq:bpsigma}. We may thus recover the well-known 
residual torsion-free nilpotence
of $\PB_n\subset \BC_n$. The proposition also applies to iterated semidirect products
of free groups  with trivial monodromy action in homology, e. g., fundamental groups 
of fiber-type arrangements of complex hyperplanes; see \cite{FR}.

\subsection{}
\label{ss56}

We are ready for the main result of this section.

\begin{theorem}
\label{thm:mcformal}
The McCool groups $\BC_n$ have the following properties.
\begin{enumerate}
\item \label{mc1}
The group $\BC_n$ is $1$-formal.
\item \label{mc2}
The Lie algebra with grading $\gr_{\Gamma}^*(\BC_n)\otimes \bbQ$ is isomorphic to the
Lie algebra $\LL_n^*$ from Definition {\em \ref{def:inforalg}}.
\item \label{mc3}
The homomorphism $R_n$ from \eqref{eq:defr} is a Malcev completion.
\item \label{mc4}
The above representation $R_n$ is faithful.
\end{enumerate}
\end{theorem}

\begin{proof}
Part \eqref{mc1}. By \eqref{eq:fpmal} and \eqref{eq:relmc}, the Malcev Lie algebra of
$\BC_n$ is isomorphic to the quotient of 
$\widehat{\bbL}\langle v_{ij} \mid 1\le i\ne j\le n \rangle$
by the relations
\[
\left \{
\begin{array}{cll}
(I) & \log ((\exp (v_{ik}), \exp (v_{jk})))\, , & \forall ~1\le i\ne j \ne k\le n\, ; \\
(II) & \log ((\exp (v_{ij}), \exp (v_{ik})\cdot \exp (v_{jk})))\, , & \forall ~1\le i\ne j \ne k\le n\, ; \\
(III) & \log ((\exp (v_{ij}), \exp (v_{kl})))\, , & \forall ~1\le i\ne j \ne k \ne l\le n\, .
\end{array}
\right .
\]

Using \cite[Lemma 2.5]{Pa2}, the above relations may be replaced by
\[
\left \{
\begin{array}{cll}
(I') & [v_{ik}, v_{jk}]\, , & \forall ~1\le i\ne j \ne k\le n\, ; \\
(II') & [v_{ij}, \log (\exp (v_{ik})\cdot \exp (v_{jk}))]\, , & \forall ~1\le i\ne j \ne k\le n\, ; \\
(III') & [v_{ij}, v_{kl}]\, , & \forall ~1\le i\ne j \ne k \ne l\le n\, .
\end{array}
\right .
\]
Due to $(I')$, $(II')$ becomes
\[
(II'') \quad [v_{ij}, v_{ik}+ v_{jk}]\, .
\]
Since all relations $(I')$, $(II'')$ and $(III')$ are quadratic, we are done.

Part \eqref{mc2}. We have seen that $E_{\BC_n}\cong \LL_n$; see Definition \ref{def:inforalg}.
By the defining property of the Malcev completion of a group (see \S \ref{ss54}),
$\gr_{\Gamma}^*(\BC_n)\otimes \bbQ\cong \gr_F^*(\LL_n)= \LL_n^*$.

Part \eqref{mc3}. It is enough to verify that 
$\gr^*(R_n)=\id \colon \gr_{\Gamma}^*(\BC_n)\otimes \bbQ \rightarrow \LL_n^*$. Both
Lie algebras being generated in degree one, it suffices to check this on generators, i. e.,
to show that $R_n (a_{ij})\equiv_2 1+ v_{ij}$, for all $i\ne j$, which is clear from
the definition of $R_n$.

Part \eqref{mc4}. Follows from the residual torsion-free nilpotence of $\BC_n$; see
Proposition \ref{prop:tore} and the discussion of Torelli groups of free groups.
\end{proof}

\begin{remark}
\label{rk=subgr}
Unlike residual properties, $1$-formality is not necessarily inherited by subgroups.
For this reason, it seems hopeless to deduce from Theorem \ref{thm:mcformal} \eqref{mc1}
an alternative simple proof for $1$-formality of pure braid groups.
\end{remark}

\begin{remark}
\label{rem:upper}
The {\em upper-triangular} McCool groups, $\BC_n^+ \subset \BC_n$, were examined
in detail in \cite{CPVW}. It follows from \cite[Sections 4--5]{CPVW} that $\BC_n^+$
is generated by $\{ a_{ij}\mid n\ge i>j\ge 1\}$, with the following sublist of 
\eqref{eq:relmc} as defining relations:
\[
(I)\quad \mbox{with}\quad i,j>k\, ;\quad (II)\quad \mbox{with}\quad i>j>k\, ; 
\quad (III)\quad \mbox{with}\quad i>j, k>l\, .
\]

The proof of Theorem \ref{thm:mcformal} applies verbatim to these groups, and gives the 
following information. The group $\BC_n^+$ is $1$-formal. The Lie algebra with grading
$\gr_{\Gamma}^*(\BC_n^+)\otimes \bbQ$ is generated by 
$\{ v_{ij}\mid n\ge i>j\ge 1\}$, with defining relations as in Definition \ref{def:inforalg},
subject to the restrictions on indices described above. (Note that the authors of 
\cite{CPVW} obtain the same presentation for $\gr_{\Gamma}^*(\BC_n^+)$, over $\bbZ$, 
by using a different method.) A Malcev completion for $\BC_n^+$, $R_n^+$, may be obtained
by completing $\gr_{\Gamma}^*(\BC_n^+)\otimes \bbQ$ with respect to the degree
filtration, and then defining $R_n^+(a_{ij})= \exp (v_{ij})$, for $i>j$. The 
representation $R_n^+$ is faithful.
\end{remark}

\section{New universal representations and finite type invariants}
\label{sec:new}

We will show that the family of representations, 
$\{ R_n\otimes \id \colon \BP_n \to \OO_n \rtimes \bbQ [\Sigma_n] \}$,
constructed in \S \ref{ss53}, shares the essential properties of a \Dd family, 
$\{ \rho_n\colon \BB_n \to \AA_n \rtimes \bbK [\Sigma_n] \}$, listed in \S \ref{ss32}.
This leads to a geometric interpretation of the family $\{ R_n\otimes \id \}$,
in terms of finite type invariants for welded braids.

\subsection{}
\label{ss61}

Denote by $\rho_n'\colon \BB_n \to \OO_n \rtimes \bbQ [\Sigma_n]$ the restriction of
$R_n\otimes \id$ to $\BB_n$. Recall from \S\S \ref{ss51}--\ref{ss52} 
that one has projections, $\BP_n \twoheadrightarrow \Sigma_n$ and
$\exp \LL_n \rtimes \Sigma_n \twoheadrightarrow \Sigma_n$. One also has
inclusions, $\BP_{n-1}\hookrightarrow \BP_n$ and
$\exp \LL_{n-1} \rtimes \Sigma_{n-1} \hookrightarrow \exp \LL_n \rtimes\Sigma_n$.
So, it makes sense to speak about properties (E), $(\Sigma)$ and (S) from
\S \ref{ss23}, for $\{ R_n\otimes \id \}$ and $\{ \rho'_n \}$.

\begin{theorem}
\label{thm:newuniv}
Both families, $\{ R_n \otimes \id \}$ and $\{ \rho'_n \}$, have the following properties.
\begin{enumerate}
\item \label{n1}
They satisfy conditions {\em (E), $(\Sigma)$} and {\em (S)}.
\item \label{n2}
They consist of faithful representations.
\item \label{n3}
They are normalized by: $R_n \otimes \id (a_{ij}) \equiv_2 1+ v_{ij}$,
for $1\le i \ne j \le n$, respectively $\rho'_n \sigma_i =u_i \otimes s_i$, where
$u_i \in \exp \LL_n$ and $u_i \equiv_2 1+ v_{i, i+1}$, for $1\le i<n$.
\end{enumerate}
\end{theorem}

\begin{proof}
Part \eqref{n1}. Property (E) was noticed at the end of \S \ref{ss53}. The other two
conditions are direct consequences of \eqref{eq:defrbp}, via our discussion 
of symmetry and stability morphisms from \S\S \ref{ss51}--\ref{ss52}.

Part \eqref{n2}. Follows from Theorem \ref{thm:mcformal}\eqref{mc4}.

Part \eqref{n3}. Use the definition of $R_n\otimes \id$ and $\rho'_n$, together with
\eqref{eq:sigmaas}.
\end{proof}

Thus, Theorem \ref{thm:intro2} from the Introduction is established.

\subsection{}
\label{ss62}

The {\em finite type} invariants for classical braids and links are defined by conditions
coming from the (iterated) application of a local move: exchanging negative and positive
crossings, in the associated diagram of a projection. 

The groups $\BP_n$ were given a geometric interpretation in \cite{FRR}. They may be
identified with {\em welded braid groups}, consisting of braids that may also have 
singular points (welds), modulo certain allowable moves. The natural local move in
this context is to replace a weld by a positive crossing.

More formally, let $J\subset \bbQ [\BP_n]$ be the two-sided ideal generated by
$\{ \sigma_i -s_i \}_{1\le i<n}$. The $J$--adic filtration $\{ J^k \}_{k\ge 0}$ will
then play the role of the {\em Vassiliev filtration} of classical braids. Let
$\{ F_k \OO_n \}_{k\ge 0}$ be the canonical complete filtration of $\OO_n$. 
Define on $\OO_n \rtimes \bbQ [\Sigma_n]$ the multiplicative filtration
$F_k := F_k \OO_n \otimes \bbQ [\Sigma_n]$, having the property that 
$F_k \cdot F_l\subset F_{k+l}$, for all $k,l$. View $R_n \otimes \id$ as an algebra map,
$R_n\otimes \id \colon \bbQ [\BP_n]\to \OO_n \rtimes \bbQ [\Sigma_n]$. Our
next result establishes that $R_n\otimes \id$ is a {\em universal invariant \` a la Vassiliev},
for welded braids.

A similar result is known for braids. Let $V \subseteq \bbQ [\BB_n]$ be the two-sided 
ideal generated by $\{ \sigma_i -\sigma_i^{-1}\}_{1\le i<n}$. Consider the $V$-adic filtration
$\{ V^k\}_{k\ge 0}$ on $\bbQ [\BB_n]$, and the multiplicative filtration
$\{ F_k \AA_n \otimes \bbQ [\Sigma_n]\}_{k\ge 0}$  on $\AA_n \rtimes \bbQ [\Sigma_n]$,
where $\{ F_k \AA_n \}$ is the canonical complete filtration of $\AA_n$. Extend the \Dd
representation $\rho_n$ to an algebra map, 
$\rho_n \colon \bbQ [\BB_n]\to \AA_n \rtimes \bbQ [\Sigma_n]$. Then $\rho_n$
respects filtrations and induces a 
degree zero multiplicative isomorphism at the associated graded level.

\begin{theorem}
\label{thm:welded}
The algebra map $R_n \otimes \id$ respects the above filtrations and induces a 
degree zero multiplicative isomorphism at the associated graded level,
$$\gr^*(R_n \otimes \id) \colon \gr^*_J (\BP_n) \stackrel{\sim}{\longrightarrow} 
\OO^*_n \rtimes \bbQ [\Sigma_n]\, .$$
\end{theorem}

Thus, the dimensions of the vector spaces dual to $\OO_n^k \otimes \bbQ [\Sigma_n]$
(respectively $\AA_n^k \otimes \bbQ [\Sigma_n]$) may be interpreted as the maximal number
of linearly independent finite type weight systems for welded braids
(respectively braids) of degree $k$, for all $k$. These numbers are known for braids,
since Kohno computed in \cite{Kohn} the Hilbert series of $\AA_n^*$. In the case of
$\OO_n^*$, the determination of the Hilbert series seems more difficult than for $\AA_n^*$. 

The strategy of proving a similar result for $\BB_n$, see \cite[Theorem 1.1]{Pa4},
may be adapted to deduce the above theorem (actually, things are here easier than in
\cite{Pa4}, since $R_n\otimes \id$ is a multiplicative map). The starting point is 
the following.

\begin{lemma}
\label{lem:start}
Set $B= \bbQ [\BP_n]$ and let $I$ be the augmentation ideal of $\BC_n$. Then
$J^k= B I^k B= B I^k =I^k B$, for all $k$.
\end{lemma}

\begin{proof}
For the last two equalities, follow the proof of Lemma 2.1 from \cite{Pa4}. 
It remains to show that $J= BIB$. Since plainly $BIB$ is the kernel of
the algebra map $\bbQ [\BP_n]\twoheadrightarrow \bbQ [\Sigma_n]$ 
(see \eqref{eq:bpsigma}), we infer that $J\subset BIB$. To obtain the other
inclusion, it is enough to check that $a_{ij}\equiv 1$ mod $J$, for
$1\le i\ne j\le n$. Since $a_{i, i+1}= \sigma_i s_i^{-1}$, see \eqref{eq:sigmaas},
we may pick $s\in \Sigma_n$ fixing $i$ and sending $i+1$ to $j$, to deduce from
\eqref{eq:equiv} that $a_{ij}= s (\sigma_i s_i^{-1}) s^{-1}\equiv 1$ mod $J$.
\end{proof}

\subsection{}
\label{ss63}

{\em Proof of Theorem} \ref{thm:welded}. Checking that $R_n \otimes\id (J^k)\subset F_k$,
for all $k$, amounts to verifying that 
$R_n \otimes \id (\sigma_i -s_i)\in F_1\OO_n \otimes \bbQ [\Sigma_n]$, for $1\le i<n$.
By the definition of $R_n$ and \eqref{eq:sigmaas}, 
$R_n \otimes \id (\sigma_i -s_i)= (\exp (v_{i, i+1})-1)\otimes s_i$, which proves
Theorem \ref{thm:welded}, except for the fact that each map $\gr^k(R_n \otimes \id)$
is a $\bbQ$-linear isomorphism.

To finish the proof, we introduce an intermediate object, namely the graded vector 
space $\gr_I^*(\BC_n):= \oplus_{k\ge 0} I^k/I^{k+1}$. Since $R_n :\BC_n \to \exp \LL_n$
is a Malcev completion, by Theorem \ref{thm:mcformal}\eqref{mc3}, the general theory
from \cite{Q} guarantees the fact that the induced map, 
$\gr^*(R_n)\colon \gr^*_I(\BC_n)\to \gr^*_F(\UU \LL_n =\OO_n)$ is an isomorphism.
Hence, we obtain a degree zero isomorphism,
$$\gr^*(R_n)\otimes \id \colon \gr^*_I(\BC_n)\otimes \bbQ [\Sigma_n]
\stackrel{\sim}{\longrightarrow} \OO_n^* \otimes \bbQ [\Sigma_n]\, .$$

The exact sequence \eqref{eq:bpsigma} gives rise to a vector space isomorphism,
$\Psi \colon \bbQ [\BC_n]\otimes \bbQ [\Sigma_n]\stackrel{\sim}{\rightarrow}
\bbQ [\BP_n]$, defined by $\Psi (c\otimes s)= c\cdot s$, for $c\in \BC_n$ and $s\in \Sigma_n$.
Due to Lemma \ref{lem:start}, the argument from \cite[\S 2.2]{Pa4} shows
that $\Psi$ identifies the filtrations $\{ I^k \otimes \bbQ [\Sigma_n] \}$ and
$\{ J^k \}$. Consequently, we obtain another degree zero isomorphism,
\[
\gr^*(\Psi) \colon \gr_I^*(\BC_n)\otimes \bbQ [\Sigma_n]\stackrel{\sim}{\longrightarrow}
\gr_J^* (\BP_n)\, .
\]

We finish by showing that 
$\gr^*(R_n \otimes \id) \circ \gr^*(\Psi)= \gr^*(R_n) \otimes \id$, which
follows at once from the definition of $R_n \otimes \id$ and $\Psi$.
\hfill $\square$

This completes the proof of Theorem \ref{thm:intro3} from the Introduction.

Since plainly $V^k \subseteq J^k$, for all $k$, we infer that the representations $\rho_n'$
(obtained by restricting $R_n \otimes \id$ to $\bbQ [\BB_n]$) give rise to 
finite type invariants for braids. At the associated graded level, the graded
algebra map $\gr_V^* (\BB_n) \to \gr_J^* (\BP_n)$, induced by the inclusion
$\BB_n \hookrightarrow \BP_n$, may be identified, via the isomorphisms provided by
$\rho_n$ and $R_n \otimes \id$, with 
$\delta_n \otimes \id : \AA_n^* \rtimes \bbQ [\Sigma_n] \to \OO_n^* \rtimes \bbQ [\Sigma_n]$.
Here, the Hopf algebra map $\delta_n : \AA_n^* \to \OO_n^*$ sends the generator $t_{ij}$
to $v_{ij}+ v_{ji}$, for any $1\le i\ne j\le n$. 

The injectivity of $\delta_n$ is equivalent to the fact that all finite type invariants
for $n$-braids come from welded braids. We have checked that this holds for $n\le 3$.
It would be interesting to know what happens in general.

\begin{ack}
We are grateful to Dan Cohen, who directed us to the work done in \cite{MC} and
\cite{CPVW}, and raised the $1$--formality question for (upper-triangular) McCool groups.
\end{ack}


\begin{thebibliography}{00}

\bibitem{Art}
E.~Artin: Theory of braids,
{\em Annals of Math.} (2) {\bf 48} (1947), 101--126.

\bibitem{Bi} J.~Birman: 
{\em Braids, links and mapping class groups}, 
Ann. of Math. Studies, vol. {\bf 82}, Princeton Univ. Press,
Princeton, NJ: 1975. 

\bibitem{B}
J.~Birman: New points of view in knot theory,
{\em Bull. Amer. Math. Soc.} (2) {\bf 28} (1993), 253--287.

\bibitem{Che}
K.--T.~Chen: Extension of $C^{\infty}$ function algebra by integrals and 
Malcev completion of $\pi_1$, 
{\em Advances in Math.} {\bf 23} (1977), 181--210.


\bibitem{Chow}
W.L.~Chow: On the algebraic braid group,
{\em Annals of Math.} {\bf 49} (1948), 654--658.

\bibitem{CPVW}
F.~R.~Cohen, J.~Pakianathan, V.~Vershinin and J.~Wu: 
Basis--conjugating automorphisms of a free group and associated Lie algebras,
preprint {\tt arxiv:math.GR/0610946}.


\bibitem{Drin}
V.~G.~Drinfel'd:
On quasitriangular quasi-Hopf algebras and a group closely connected with
$ Gal(\overline{\bbQ}/\bbQ) $,
{\em Leningrad Math. J.} (4) {\bf 2} (1991), 829--860.

\bibitem{FR}
M. Falk, R. Randell:
The lower central series of a fiber-type arrangement,
{\em Invent. Math.} {\bf 82} (1985), 77--88.

\bibitem{FRR}
R.~Fenn, R.~Rimanyi and C.~Rourke:
The braid--permutation group,
{\em Topology} (1) {\bf 36} (1997), 123--135.

\bibitem{H}
R.~Hain: Infinitesimal presentations of the Torelli groups,
{\em J. Amer. Math. Soc.} {\bf 10} (1997), 597--651.


\bibitem{Kas}
C.~Kassel: {\em Quantum Groups},
Graduate Texts in Mathematics, vol. {\bf 155},
Springer-Verlag, Berlin: 1995.

\bibitem{Kohn}
T.~Kohno: S\'erie de Poincar\'e-Koszul associ\'ee aux groupes de tresses pures,
{\em Invent. Math.} {\bf 82} (1985), 57--75.

\bibitem{K2}
T.~Kohno: Monodromy representations of braid groups and Yang--Baxter equations, 
{\em Ann. Inst. Fourier} {\bf 37} (1987), 139--160.

\bibitem{MKS}
W.~Magnus, A.~Karrass and D.~Solitar:
{\em Combinatorial Group Theory (2nd ed.)}, Dover, New~York: 1976.

\bibitem{MC}
J.~McCool: On basis--conjugating automorphisms of free groups,
{\em Canadian J. Math.} (6) {\bf 38} (1986), 1525--1529.


\bibitem{Mor}
S.~Moran: {\em The Mathematical Theory of Knots and Braids},
North--Holland Mathematics Studies, vol. {\bf 82}, 
Elsevier, Amsterdam: 1983.

\bibitem{OS} 
P.~Orlik and L.~Solomon:
Combinatorics and topology of complements of hyperplanes, 
{\em Invent. Math.} {\bf 56} (1980), 167--189.


\bibitem{Pa1} S.~Papadima:
Finite determinacy phenomena for finitely presented groups,
in: {\em Proceedings of the 2nd Gauss Symposium. Conference A:
Mathematics
and Theoretical Physics (Munich, 1993)}, pp.~507--528, Sympos. Gaussiana,
de Gruyter, Berlin: 1995.

\bibitem{Pa2} S.~Papadima:
Campbell--Hausdorff invariants of links,
{\em Proc. London Math. Soc.} (3) {\bf 75} (1997), 641--670.

\bibitem{Pa3} S.~Papadima:
Braid commutators and homogenous Campbell--Hausdorff tests,
{\em Pacific Journ. of Math.} {\bf 197} (2001), 383--416.

\bibitem{Pa4} S.~Papadima:
The universal finite--type invariant of braids, with integer
coefficients, 
{\em Topology and its Applications} {\bf 118} (2002), 169--185.


\bibitem{Piun} S.~Piunikhin:
Combinatorial expression for universal Vassiliev link invariant,
{\em Commun. Math. Phys.} {\bf 168} (1995), 1--22.

\bibitem{Q} D.~Quillen: Rational homotopy theory,
{\em Annals of Math.} {\bf 90} (1969), 205--295.

\bibitem{S}
D.~Sullivan: Infinitesimal computations in topology,
{\em Inst. Hautes \'{E}tudes Sci. Publ. Math.}
{\bf 47} (1977), 269--331.


\end{thebibliography}
\end{document}